\documentclass[11pt,reqno]{amsart}
\usepackage{amssymb,latexsym}
\usepackage{graphicx}
\usepackage{mathtools}
\usepackage{color}
\usepackage[hyphens]{url}
\usepackage[T1]{fontenc}
\usepackage{hyperref}
\usepackage{amsmath}
\usepackage{mathdots}
\usepackage{breqn}
\usepackage[toc,page]{appendix}
\usepackage{url}    
\usepackage{breqn}
\usepackage{hyperref}
\usepackage{breakurl}
\newcommand{\bburl}[1]{\textcolor{blue}{\url{#1}}}

\calclayout

\makeatletter
\newcommand{\monthyear}[1]{%
  \def\@monthyear{\uppercase{#1}}}
\newcommand{\volnumber}[1]{%
  \def\@volnumber{\uppercase{#1}}}
\AtBeginDocument{%
\def\ps@plain{\ps@empty
  \def\@oddfoot{\@monthyear \hfil \thepage}%
  \def\@evenfoot{\thepage \hfil \@volnumber}}
\def\ps@firstpage{\ps@plain}
\def\ps@headings{\ps@empty
  \def\@evenhead{%
    \setTrue{runhead}%
    \def\thanks{\protect\thanks@warning}%
    \uppercase{\ }\hfil}%
  \def\@oddhead{%
    \setTrue{runhead}%
    \def\thanks{\protect\thanks@warning}%
    \hfill\uppercase{On Zeckendorf Related Partitions Using the Lucas Sequence}}%
  \let\@mkboth\markboth
  \def\@evenfoot{%
    \thepage \hfil \@volnumber}%
  \def\@oddfoot{%
    \@monthyear \hfil \thepage}%
  }%
\footskip=25pt
\pagestyle{headings}%
}
\makeatother

\theoremstyle{plain}
\numberwithin{equation}{section}
\newtheorem{thm}{Theorem}[section]

\newcommand{\ignore}[1]{}

\newcommand\be{\begin{eqnarray}}
\newcommand\ee{\end{eqnarray}}
\newcommand\bea{\begin{eqnarray}}
\newcommand\eea{\end{eqnarray}}
\newcommand\ben{\begin{enumerate}}
\newcommand\een{\end{enumerate}}

\newtheorem{lem}[thm]{Lemma}

\newtheorem{defi}[thm]{Definition}

\newtheorem{rek}[thm]{Remark}

\begin{document}

\monthyear{2021}
\volnumber{}
\setcounter{page}{1}
\title{On Zeckendorf Related Partitions Using the Lucas Sequence}

\author{H\`ung Vi\d{\^e}t Chu, David C. Luo, and Steven J. Miller}

\address{Department of Mathematics, University of Illinois at Urbana-Champaign, Urbana, IL 61820} \email{hungchu2@illinois.edu}

\address{Department of Mathematics, Emory University, 400 Dowman Dr., Atlanta, GA 30322} \email{dluo6745@gmail.com}

\address{Department of Mathematics and Statistics, Williams College, Williamstown, MA 01267} \email{sjm1@williams.edu}

\date{\today}

\begin{abstract}
Zeckendorf proved that every natural number has a unique partition as a sum of non-consecutive Fibonacci numbers. Similarly, every natural number can be partitioned into a sum of non-consecutive terms of the Lucas sequence, although such partitions need not be unique. In this paper, we 
\begin{itemize}
    \item [(1)] prove that a natural number can have at most two distinct non-consecutive partitions in the Lucas sequence,
    \item [(2)] find all natural numbers with a fixed term in their partition, and
    \item [(3)] calculate the limiting value of the proportion of natural numbers that are not uniquely partitioned into the sum of non-consecutive terms in the Lucas sequence. 
\end{itemize}
\end{abstract}

\thanks{The authors thank Curtis D.~Herink and David Zureick-Brown for helpful conversations, Jeffrey Shallit for pointing out a gap in reasoning in an earlier version, the anonymous referee for useful comments, and Elvin Gu for coding support. The third author was supported by NSF grants DMS1561945.}

\maketitle
\section{Introduction} 
The Fibonacci numbers have fascinated mathematicians for centuries with many interesting properties. By convention, the Fibonacci sequence $\left\{F_n\right\}_{n=0}^{\infty}$ is defined as follows: let $F_0 = 0$, $F_1 = 1$, and $F_n = F_{n-1} + F_{n-2}$, for $n\ge 2$.
A beautiful theorem of Zeckendorf \cite{Z} states that every natural number $n$ can be uniquely written as a sum of non-consecutive Fibonacci numbers. This gives the so-called Zeckendorf partition of $n$. A formal statement of Zeckendorf's theorem is as follows:
\begin{thm}[Zeckendorf]\label{p1}
For any $n\in\mathbb{N}$, there exists a unique increasing sequence of natural numbers $\{c_1, c_2, \ldots, c_k\}$ such that $c_1\ge 2$, $c_i\ge c_{i-1}+2$ for $i = 2, 3, \ldots, k$, and $n = \sum_{i=1}^kF_{c_i}$. 
\end{thm}
Much work has been done to understand the structure of Zeckendorf partitions and their applications (see \cite{BDEMMTW1, BDEMMTW2, B, CSH, Fr, MG1, MG2, HS, L, MW1, MW2}) and to generalize them (see \cite{Chu, Day, DDKMMV, DFFHMPP, FGNPT, GTNP, Ho, K, ML, MMMS, MMMMS}). In this paper, we study the partition of natural numbers into Lucas numbers. The Lucas sequence $\left\{L_n\right\}_{n=0}^{\infty}$ is defined as follows: let $L_0 = 2$, $L_1 = 1$, and $L_n = L_{n-1} + L_{n-2}$, for $n\ge 2$. As the Lucas sequence is closely related to the Fibonacci sequence, it is not surprising that we can also partition natural numbers using Lucas numbers. 

\begin{thm}[Zeckendorf]\label{p2}
    Every natural number can be partitioned into the sum of non-consecutive terms of the Lucas sequence.
\end{thm}

Note that the distinction between Theorems \ref{p1} and \ref{p2} lies in the \textit{uniqueness} property of such partitions of natural numbers in the Fibonacci and Lucas sequences.~Although $5$ is uniquely partitioned into $F_5 = 5$ in $\{F_2, F_3, \ldots \}$, its partition is not unique in the Lucas sequence as $5 = L_0 + L_2 = 2 + 3$ and $5 = L_1 + L_3 = 1 + 4$. In \cite{B2}, Brown shows various ways to have a unique partition using Lucas sequence. \footnote{For more on Brown's criteria, see \cite{BHLMT1, BHLMT2}.} In this paper, we prove the following results.

\begin{thm}\label{m1}
If we allow $L_0$ and $L_2$ to appear simultaneously in a partition, each natural number can have at most two distinct non-consecutive partitions in the Lucas sequence.
\end{thm}

\begin{thm}\label{m2} Suppose that we do not allow $L_0$ and $L_2$ to appear simultaneously in a partition. 
The set of all natural numbers having the summand $L_k$ in their partition is given by 
$$Z(k) \ =\ \begin{cases}\left\{2+3n+\left\lfloor\frac{n+1}{\Phi}\right\rfloor\, :\, n\ge 0\right\}, & {\rm if}~ k = 0,\\
\left\{3n+\left\lfloor\frac{n+\Phi^2}{\Phi}\right\rfloor\, :\, n\ge 0\right\}, &{\rm if}~ k = 1,\\
\left\{L_k\left\lfloor\frac{n+\Phi^2}{\Phi}\right\rfloor+ nL_{k+1}+j\, :\, n\ge 0 ~{\rm and}~  0\le j\le L_{k-1}-1\right\}, &{\rm if}~ k \ge 2.
\end{cases}$$
\end{thm}
Theorem \ref{m2} is an analogue of \cite[Theorem 3.4]{MG1}. For $k\ge 0$, we find all natural numbers having the summand $L_k$ in their partition. We have a different formula when $k = 0$ instead of one formula for all values of $k$ as in \cite[Theorem 3.4]{MG1}. 

Our next result is predicted by \cite[Theorem 1]{Cha}, which deals with general recurrence relations; however, in the case of Lucas numbers, we can relate Lucas partitions to the golden string.

\begin{thm}\label{m3}
If we allow $L_0$ and $L_2$ to appear simultaneously in a partition, the proportion of natural numbers that are not uniquely partitioned into the sum of non-consecutive terms of the Lucas sequence converges to $\frac{1}{3\Phi+1}$, where $\Phi$ is the golden ratio. 
\end{thm}

\section{Preliminaries}
\subsection{Definitions}
\begin{defi}
Let $A = \left\{a_0, a_1, \ldots, a_m \right\}$ be the set consisting of the first $m+1$ terms of the sequence $\big\{a_k\big\}_{k=0}^{\infty}$.~We say a proper subset $B$ of $A$ is a \textit{non-consecutive subset} of $A$ if the elements of $B$ are pairwise non-consecutive in $\big\{a_k\big\}_{k=0}^{\infty}$.~Furthermore, we say a sum $S$ is a \textit{non-consecutive sum} of $A$ if $S$ is the sum of distinct elements of $A$ that are pairwise non-consecutive in $\big\{a_k\big\}_{k=0}^{\infty}$. 
\end{defi}
\begin{defi}
Let $A_m = \left\{L_0, L_1, \ldots, L_m \right\}$ denote the set consisting of the first $m+1$ terms of the Lucas sequence.
\end{defi}

\subsection{The golden string}
The golden string $S = BABBABABBABBA\ldots$ is defined to be the infinite string of $A$'s and $B$'s constructed recursively as follows. Let $S_1 = A$ and $S_2 = B$, and then, for $k\ge 3$, $S_k$ is the concatenation of $S_{k-1}$ and $S_{k-2}$, which we denote by $S_{k-1}\circ S_{k-2}$. For example, $S_3 = S_2\circ S_1 = a_2\circ a_1 = BA$, $S_4 = S_3\circ S_2 = a_2a_1\circ a_2 = BAB$, $S_5  = S_4\circ S_3  = BABBA$, and so on. Interestingly, the golden string is highly connected to the Zeckendorf partition \cite{MG2}. As we will see later, the string is also closely related to the partitions of natural numbers into Lucas numbers. 
\begin{rek}\label{r1}\normalfont We mention two properties of the golden string that we will use in due course. 
\begin{itemize}
\item[(1)] For $j\ge 1$, the $(F_{2j})$th character of $S$ is $B$ and the $(F_{2j+1})$th character of $S$ is $A$. This can be easily proved using induction. 
\item[(2)] The number of $B$'s amongst the first $n$ characters of $S$ is given by 
$\left\lfloor\frac{n+1}{\Phi}\right\rfloor$,
where $\Phi = \frac{1+\sqrt{5}}{2}$ is the golden ratio. For a proof of this result, see \cite[Lemma 3.3]{MG2}.
\end{itemize}
\end{rek}

\section{At Most Two Partitions}

In this section, we present our results that determine the maximum number of non-consecutive partitions that a natural number can have in the Lucas sequence, the proofs of which are adapted from \cite{HR}. Before we prove Theorem \ref{m1}, we introduce the following preliminary lemmas. For the proofs of Lemmas \ref{Lemma3} and \ref{Lemma5}, see Appendix B.

\begin{lem}\label{Lemma3}
Let $S$ be any non-consecutive sum of $A_m$.~Then
    \begin{enumerate}
            \item if $m$ is odd, $S$ assumes all values from 0 to $L_{m+1}-1$ inclusive, and
            \item if $m$ is even, then $S$ assumes all values from 0 to $L_{m+1}+1$ inclusive, excluding $L_{m+1}$.
    \end{enumerate}
\end{lem}

\begin{lem}\label{Lemma5}
If $m \geq 0$, then $L_{2m+1}+1$ has exactly two non-consecutive partitions in the Lucas sequence.
\end{lem}
\begin{proof}[Proof of Theorem \ref{m1}]
It suffices to show that for every non-negative integer $m$, there is no natural number that is equal to three or more distinct non-consecutive sums of $A_m$.~We proceed by strong induction.~No natural is equal to three or more distinct non-consecutive sums of $A_0$ and $A_1$.~This shows the base case.~Assume Theorem \ref{m1} holds for all non-negative integers less than or equal to $m=k$.~In our first case, suppose that $k$ is odd.~From Lemma \ref{Lemma3}, the non-consecutive sums that we can form from $A_k$ are the values from 0 to $L_{k+1}-1$ inclusive.~Hence, when we add the term $L_{k+1}$ to $A_k$, all new non-consecutive sums that can be formed must be at least $L_{k+1}$.~This implies there is no possible way in which we can form a third distinct non-consecutive sum of $A_{k+1}$ for any natural number because there is no intersection between the non-consecutive sums in which we can form before and after the addition of the term $L_{k+1}$.~When $k \geq 2$ is even, we have from Lemma \ref{Lemma3} that all non-consecutive sums we can form from $A_k$ are the values from 0 to $L_{k+1}+1$ inclusive, excluding $L_{k+1}$.~When we add the term $L_{k+1}$ to $A_k$, all new non-consecutive sums that can be formed are at least $L_{k+1}$ with $L_{k+1}+1$ being the only non-consecutive sum formed again, namely $L_{k+1}+L_1$.~By Lemma \ref{Lemma5}, we know that $L_{k+1}+1$ has exactly two distinct non-consecutive partitions in the Lucas sequence.~Therefore, there is no possible way in which we can form a third distinct non-consecutive sum of $A_{k+1}$ for any natural number.~This completes the inductive step.
\end{proof}

\section{Partitions with a Fixed Term}

Let $\mathcal{X}_k$ denote the set of all natural numbers having $L_k$ as the smallest summand in their partition. Let $\mathcal{Q}_k = \{q_k(j)\}_{j\ge 1}$ be the strictly increasing sequence obtained by rearranging the elements of $\mathcal{X}_k$ into ascending numerical order. We consider the cases $k = 0$ and $k\ge 1$ separately.

\subsection{When $k = 0$}

Table 1 replaces each term $q_k(j)$ in $\mathcal{Q}_k$ with an ordered list of the summands in its partition. 

\begin{center}
\begin{tabular}{cccccccccccc}
 Row  &               &          &     &        &    &         &     &        &  &          &    \\
 \hline
1     &               & $L_0$    &     &        &    &         &     &        &  &         &     \\
2     &               & $L_0$    &     & $L_3$  &    &         &     &        &  &         &    \\
3     &               & $L_0$    &     &        &    & $L_4$   &     &        &  &         &     \\
$4$   &               & $L_0$    &     &        &    &         &     & $L_5$  &  &         &     \\
$5$   &               & $L_0$    &     & $L_3$  &    &         &     & $L_5$  &  &         &       \\
$6$   &               & $L_0$    &     &        &    &         &     &        &  &  $L_6$  &   \\
$7$   &               & $L_0$    &     & $L_3$  &    &         &     &        &  &  $L_6$  &     \\
$8$   &               & $L_0$    &     &        &    & $L_4$   &     &        &  &  $L_6$  &    
\end{tabular}
\end{center}

\begin{center}
Table 1. The partitions of the natural numbers having $L_0$ as their smallest summand.  
\end{center}

\begin{lem}\label{l1}
For $j\ge 3$, the rows of Table 1 for which $L_{j}$ is the largest summand are those numbered from $F_{j-1}+1$ to $F_{j}$ inclusive. 
\end{lem}

\begin{proof}
We prove by induction. \textit{Base cases:} it is easy to check that the statement of the lemma is true for $j = 3$ and $j = 4$. \textit{Inductive hypothesis:} assume that it is true for all $j$ such that $3\le j\le m$ for some $m\ge 4$. By the inductive hypothesis, the number of rows such that their largest summands are no greater than $L_{m-1}$ is
$$1+\sum_{j=3}^{m-1}(F_{j} - F_{j-1}) \ =\ F_{m-1},$$
which is also the number of rows whose largest summand is $L_{m+1}$. Due to the inductive hypothesis, the rows whose largest summand is $L_m$ are numbered from $F_{m-1}+1$ to $F_{m}$ inclusive. Therefore, the rows whose largest summand is $L_{m+1}$ are numbered from $F_{m}+1$ to $F_{m+1}$, as desired. This completes our proof. 
\end{proof}

\begin{lem}\label{k1}
For $j\ge 1$, we have
$$q_k(j+1)-q_k(j) \ =\ \begin{cases}L_{2}, & \mbox{ if }A \mbox{ is the }j\mbox{th character of }S,\\
L_{3}, &\mbox{ if }B \mbox{ is the }j\mbox{th character of }S.\end{cases}$$
\end{lem}

\begin{proof}
We prove by induction. \textit{Base cases:} it is easy to check that the statement of the lemma is true for $1\le j\le F_4-1$. \textit{Inductive hypothesis:} suppose that it is true for $1\le j\le F_m-1$ for some $m\ge 4$. By Lemma \ref{l1}, the number of rows in Table 1 whose largest summand is no greater than $L_{m-1}$ is 
$$1+ \sum_{j=3}^{m-1}(F_{j} - F_{j-1}) \ =\ F_{m-1},$$
which is also the number of rows whose largest summand is $L_{m+1}$. Furthermore, the rows for which $L_{m+1}$ is the largest summand are numbered from $F_{m}+1$ to $F_{m+1}$ inclusive. Therefore, the ordering of the rows in Table 1 implies that 
$q_k(i+F_{m})\ =\ q_k(i) + L_{m+1}$,
for $1\le i\le F_{m-1}$. Hence, for $1\le i\le F_{m-1}-1$, we have 
\begin{align*}
    q_k(i+1+F_{m}) - q_k(i+F_{m}) &\ =\ (q_k(i+1) + L_{m+1})-(q_k(i) + L_{m+1})\ =\ q_k(i+1) - q_k(i).
\end{align*}
By the construction of $S$, the substring comprising of its first $F_{m-1}$ characters is identical to the substring of its characters numbered from $F_{m}+1$ to $F_{m+1}$ inclusive. Thus the lemma is true for $F_{m}+1\le j\le F_{m+1}-1$. It remains to show that it is true for $j = F_{m}$. We have
\begin{align*}
    q_k(F_m+1) - q_k(F_m) \ =\begin{cases}\ L_{m+1} - (L_{m}+L_{m-2} + \cdots + L_4) \ =\ L_3,&\mbox{ if }m\mbox{ is even,}\\
    L_{m+1} - (L_m + L_{m-2} + \cdots + L_3)\ =\ L_2,&\mbox{ if }m\mbox{ is odd.}\end{cases}
\end{align*}
By Remark \ref{r1} item (1), we know that the lemma is true for $j = F_{m}$, completing the proof.  
\end{proof}

\subsection{When $k\ge 1$}

Table 2 replaces each term $q_k(j)$ in $\mathcal{Q}_k$ with an ordered list of the summands in its partition. 
\begin{center}
\begin{tabular}{cccccccccccc}
 Row  &               &          &     &        &    &         &     &        &  &          &    \\
 \hline
1     &               & $L_k$    &     &        &    &         &     &        &  &         &     \\
2     &               & $L_k$    &     & $L_{k+2}$  &    &         &     &        &  &         &    \\
3     &               & $L_k$    &     &        &    & $L_{k+3}$   &     &        &  &         &     \\
$4$   &               & $L_k$    &     &        &    &         &     & $L_{k+4}$  &  &         &     \\
$5$   &               & $L_k$    &     & $L_{k+2}$  &    &         &     & $L_{k+4}$  &  &         &       \\
$6$   &               & $L_k$    &     &        &    &         &     &        &  &  $L_{k+5}$  &   \\
$7$   &               & $L_k$    &     & $L_{k+2}$  &    &         &     &        &  &  $L_{k+5}$  &     \\
$8$   &               & $L_k$    &     &        &    & $L_{k+3}$   &     &        &  &  $L_{k+5}$  &     
\end{tabular}
\end{center}
\begin{center}
Table 2. The partitions of the natural numbers having $L_k$ as their smallest summand.  
\end{center}

Table 2 is similar to Table 1 in \cite{MG1}. The next lemma follows from \cite[Lemma 3.1]{MG1}. 

\begin{lem}\label{l2}
For $j\ge 2$, the rows of Table 2 for which $L_{k+j}$ is the largest summand are those numbered from $F_{j}+1$ to $F_{j+1}$ inclusive. 
\end{lem}

\begin{lem}
For $j\ge 1$, we have
$$q_k(j+1)-q_k(j) \ =\ \begin{cases}L_{k+1}, & \mbox{ if }A \mbox{ is the }j\mbox{th character of }S,\\
L_{k+2}, &\mbox{ if }B \mbox{ is the }j\mbox{th character of }S.\end{cases}$$
\end{lem}

\begin{proof}
We prove by induction. \textit{Base cases:} it is easy to check that the statement of the lemma is true for $j$ such that $1\le j\le F_4-1$. \textit{Inductive hypothesis:} assume that it is true for $1\le j\le F_m-1$ for some $m\ge 4$. From Lemma \ref{l2}, the first $F_{m-1}$ rows of Table 2 are those for which the largest summand is no greater than $L_{k+m-2}$. Also, the rows for which $L_{k+m}$ is the largest summand are those numbered from $F_m+1$ to $F_{m+1}$ inclusive. Therefore, the ordering of the rows implies that 
$q_k(i+F_m) \ =\ q_k(i) + L_{k+m}$,
for $i = 1, 2, \ldots, F_{m-1}$. Hence, for $i = 1, 2, \ldots, F_{m-1}-1$, we have
\begin{align*}
    q_k(i+1+F_m) - q_k(i+F_m) &\ =\ (q_k(i+1)+L_{k+m}) - (q_k(i) + L_{k+m})\ =\ q_k(i+1) - q_k(i).
\end{align*}
By the construction of $S$, the substring comprising its first $F_{m-1}$ characters is identical to the substring of its characters numbered from $F_m+1$ to $F_{m+1}$ inclusive. Thus, the lemma is true for $F_m+1\le j\le F_{m+1}-1$. It remains to show that the lemma is true for $j = F_m$. We have
\begin{align*}
    q_k(F_m+1) - q_k(F_m)\ =\begin{cases}\ L_{k+m} - (L_{k+m-1}+L_{k+m-3} + \cdots + L_{k+3}) \ =\ L_{k+2},&\mbox{ if }m\mbox{ is even,}\\
   L_{k+m} - (L_{k+m-1}+L_{k+m-3} + \cdots + L_{k+2})\ =\ L_{k+1},&\mbox{ if }m\mbox{ is odd.}\end{cases}
\end{align*}
By Remark \ref{r1} item (1), we know that the lemma is true for $j = F_{m}$, completing the proof.  
\end{proof}

We are ready to prove Theorem \ref{m2}.

\begin{proof}[Proof of Theorem \ref{m2}] We consider three cases. 
\mbox{ }\newline

Case 1: $k = 0$. By Lemma \ref{k1}, we have $\mathcal{X}_0 \ =\ \{2+a(n)L_2 + b(n)L_3: n\ge 0\}$, where $a(n)$ and $b(n)$ denote the number of $A$'s and $B$'s, respectively, amongst the first $n$ characters in the golden string. Using Remark \ref{r1} item (2), we have
\begin{align*}
    \mathcal{X}_0 &\ =\ \left\{2+3\left(n-\left\lfloor \frac{n+1}{\Phi}\right\rfloor\right)+4\left\lfloor \frac{n+1}{\Phi}\right\rfloor\, :\, n\ge 0\right\}\ =\ \left\{2+3n+\left\lfloor \frac{n+1}{\Phi}\right\rfloor\,:\, n\ge 0\right\}.
\end{align*}
It is clear that $Z(0) = \mathcal{X}_0$; hence, the statement of the lemma is true when $k = 0$. 
\mbox{ }\newline

Case 2: $k = 1$. Using a similar reasoning as above, we have
\begin{align*}
    \mathcal{X}_1 &\ =\ \left\{1+L_2\left(n-\left\lfloor \frac{n+1}{\Phi}\right\rfloor\right)+L_3\left\lfloor \frac{n+1}{\Phi}\right\rfloor\, :\, n\ge 0\right\}\\
    &\ =\ \left\{1+3\left(n-\left\lfloor \frac{n+1}{\Phi}\right\rfloor\right)+4\left\lfloor \frac{n+1}{\Phi}\right\rfloor\, :\, n\ge 0\right\}\ =\ \left\{3n+\left\lfloor \frac{n+\Phi^2}{\Phi}\right\rfloor\, :\, n\ge 0\right\}.
\end{align*}
It is clear that $Z(1) = \mathcal{X}_1$; hence, the statement of the lemma is true when $k = 1$. 
\mbox{ }\newline

Case 3: $k\ge 2$. Using a similar reasoning as above, we have
\begin{align*}
    \mathcal{X}_k &\ =\ \left\{L_k+L_{k+1}\left(n-\left\lfloor \frac{n+1}{\Phi}\right\rfloor\right)+L_{k+2}\left\lfloor \frac{n+1}{\Phi}\right\rfloor\, :\, n\ge 0\right\}\\
    &\ =\ \left\{L_k\left(1+\left\lfloor \frac{n+1}{\Phi}\right\rfloor\right)+nL_{k+1}\, :\, n\ge 0\right\}\ =\ \left\{L_k\left\lfloor \frac{n+\Phi^2}{\Phi}\right\rfloor+nL_{k+1}\, :\, n\ge 0\right\}.
\end{align*}
If $k\ge 3$, the numbers in $\{L_0, L_1, \ldots, L_{k-2}\}$ are used to obtain the partitions of all integers for which the largest summand is no greater than $L_{k-2}$. In particular, such partitions generate all integers from $1$ to $L_{k-1}-1$ inclusive. Furthermore, such partitions can be appended to any partition having $L_k$ as its smallest summand to produce another partition. Therefore,
\begin{align*}
    Z(k) \ =\ \left\{L_k\left\lfloor \frac{n+\Phi^2}{\Phi}\right\rfloor+nL_{k+1} + j\, :\, n\ge 0\mbox{ and } 0\le j\le L_{k-1}-1\right\},
\end{align*}
as desired. It is easy to check that this formula is also true for $k = 2$. 
\end{proof}

\section{Proportion of Nonunique Partitions}
Let $c(N)$ count the number of numbers that are not uniquely represented in the Lucas sequence and are at most $N$. We want to show that $\displaystyle\lim_{N\rightarrow \infty}\frac{c(N)}{N} = \frac{1}{1+3\Phi}$, where $\Phi = \frac{1+\sqrt{5}}{2}$ is the golden ratio. Note that \cite[Lemma 3]{B2} says we can make the Lucas partition unique by requiring that not both $L_0$ and $L_2$ appear in the partition. Therefore, if a number has two partitions, then one of the partition starts with $L_0+L_2$. If we can characterize all of these numbers and find a formula for $c(N)$ in terms of $N$, we are done. Call the set of these numbers $K$. We form the following table listing all of such numbers in increasing order. Let $q_k(j)$ be the $jth$ smallest number in $K$. 

\begin{center}
\begin{tabular}{cccccccccccc}
 Row  &               &          &     &        &    &         &     &        &  &          &    \\
 \hline
1     &               & $L_0+L_2$    &     &        &    &         &     &        &  &         &     \\
2     &               & $L_0+L_2$    &     & $L_4$  &    &         &     &        &  &         &    \\
3     &               & $L_0+L_2$    &     &        &    & $L_5$   &     &        &  &         &     \\
$4$   &               & $L_0+L_2$    &     &        &    &         &     & $L_6$  &  &         &     \\
$5$   &               & $L_0+L_2$    &     & $L_4$  &    &         &     & $L_6$  &  &         &       \\
$6$   &               & $L_0+L_2$    &     &        &    &         &     &        &  &  $L_7$  &   \\
$7$   &               & $L_0+L_2$    &     & $L_4$  &    &         &     &        &  &  $L_7$  &     \\
$8$   &               & $L_0+L_2$    &     &        &    & $L_5$   &     &        &  &  $L_7$  &    
\end{tabular}
\end{center}

\begin{center}
Table 3. The partitions of the natural numbers having $L_0$ and $L_2$ as their smallest summands.  
\end{center}

Observe that Table 3 has the same structure as Table 1. Therefore, Lemma \ref{k1} applies with a change of index. In particular, we have the following.

\begin{lem}\label{k2}
For $j\ge 1$, we have
$$q_k(j+1)-q_k(j) \ =\ \begin{cases}L_{3}, & \mbox{ if }A \mbox{ is the }j\mbox{th character of }S,\\
L_{4}, &\mbox{ if }B \mbox{ is the }j\mbox{th character of }S.\end{cases}$$
\end{lem}
Therefore, we can write 
$$K \ =\ \left\{L_0+L_2 + a(n)L_3 + b(n)L_4: n\ge 0\right\},$$
where $a(n)$ and $b(n)$ denote the number of $A$'s and $B$'s, respectively, amongst the first $n$ characters in the golden string. Hence, 
\begin{align*}K &\ =\ \left\{5 + 4\left(n-\left\lfloor{\frac{n+1}{\Phi}}\right\rfloor\right)+7\left\lfloor{\frac{n+1}{\Phi}}\right\rfloor\,:\, n\ge 0 \right\}\ =\ \left\{5+4n+3\left\lfloor{\frac{n+1}{\Phi}}\right\rfloor\,:\, n\ge 0\right\}.\end{align*}

Now, we are ready to compute the limit. 

\begin{proof}[Proof of Theorem \ref{m3}]The number of integers with two partitions up to a number $N$ is exactly
$\#\left\{n\ge 0\,:\, 5+4n+3\left\lfloor{\frac{n+1}{\Phi}}\right\rfloor\le N\right\}$.
The number is found to be $\frac{N-1}{4+\frac{3}{\Phi}}$ within an error of at most $1$. Therefore, as claimed, the limit is 
$$\lim_{N\rightarrow \infty}\frac{1}{N}\frac{N-1}{4+\frac{3}{\Phi}}\ =\ \frac{1}{4+\frac{3}{\Phi}} \ =\ \frac{1}{1+3\Phi}.$$
\end{proof}

Among the first $N$ natural numbers, we see how $\alpha = \frac{1}{3\Phi + 1} \approx 0.17082$ estimates the proportion of natural numbers within this range that do not have unique non-consecutive partitions in the Lucas sequence.~The data we collect is shown in Table 4.
\begin{center}
    \begin{tabular}{|c|c|c|}
    \hline
      N & $c\left(N \right)$ & $\beta\left(N\right)$ \\
      \hline
      10 & 1 & 10.000 \% \\ \hline
      100 & 17 & 17.000\% \\ \hline
      1,000 & 171 & 17.100\% \\ \hline
      10,000 & 1,708 & 17.080\% \\ \hline
      $10^5$ & 17,082 & 17.082\%\\ \hline
      $10^6$ & 170,820 & 17.082\%\\ \hline
    \end{tabular}
\end{center}

\begin{center}
Table 4. Proportion $\beta\left(N\right)$ of the first $N$ natural numbers that do not have unique non-consecutive partitions in the Lucas sequence.
\end{center}

\appendix

\section{Java Code}\label{AppendixA}
The following is our Java code for calculating non-consecutive partitions of natural numbers in any infinite integer sequence given by a second-order linear recurrence.~It is available on github at
\sloppy \bburl{https://github.com/dluo6745/Zeckendorf-Partitions/blob/master/ZP.java}.~For each natural number $n$ from 1 to $N$, the code returns the non-consecutive partition(s) of $n$ as a list of integers that correspond to the indices of the terms in the second-order linear recurrence sequence we are enumerating.~Furthermore, the code also returns the number of natural numbers from 1 to $N$ that do not have unique non-consecutive partitions.

\section{Proofs of Lemmas}\label{AppendixB}
\begin{proof}[Proof of Lemma \ref{Lemma3}]
We proceed by strong induction.~The non-consecutive sums that we can form from $A_0$ are 0 and $L_{1}+1$ because the empty set results in a sum of 0 and the non-consecutive sums that we can form from $A_1$ are $0$, $L_1$, and $L_{2}-1$.~This shows the base case.~Assume Lemma \ref{Lemma3} holds for all non-negative integers less than or equal to $m = k$.~Without loss of generality, suppose that $k$ is odd.~To find the range of non-consecutive sums that we can form from $A_{k+1}$, we consider the subset $A_{k+1} - \{L_k\}$.~From our inductive hypothesis, the non-consecutive sums that we can form from $A_{k-1}$ are the values from 0 to $L_k+1$ inclusive, excluding $L_k$.~By adding $L_{k+1}$ to these values, we have the following non-consecutive sums that we can form from $A_{k+1}$ range from 0 to $L_{k+2}+1$ inclusive.

To show that $L_{k+2}$ cannot be formed as a non-consecutive sum of $A_{k+1}$, we first prove a general result.~Let $B$ be a non-consecutive subset of $A_{2j}$, where $j$ is a non-negative integer such that $2j < k$.~For sake of contradiction, suppose that the sum of the elements of $B$ is equal to $L_{2j+1}$.~In our first case, suppose that $L_{2j}$ is not in $B$.~This implies $B$ is a non-consecutive subset of $A_{2j-1}$ and that the sum of the elements of $B$ is less than or equal to $L_{2j+1}-1$ from our inductive hypothesis.~Hence, we have a contradiction which implies $B$ contains the term $L_{2j}$.~Consider the set $B'= B - \{L_{2j}\}$, which is a non-consecutive subset of $A_{2j-2}$.~Because the sum of the elements of $B'$ is equal to the difference between the sum of the elements of $B$ and $L_{2j}$, this implies that the sum of the elements of $B'$ is equal to $L_{2j-1}$, which cannot be formed as a non-consecutive sum of $A_{2j-2}$ by our inductive hypothesis.~Therefore, we have a contradiction and $L_{2j+1}$ cannot be formed as a non-consecutive sum of $A_{2j}$.

Applying this result to our inductive step, we have that $L_k$ cannot be formed as a non-consecutive sum of $A_{k-1}$.~This implies there is no possible way to form $L_{k+2} = L_{k} + L_{k+1}$ as a non-consecutive sum of $A_{k+1} - \{L_k\}$.~From our inductive hypothesis, the maximum possible sum we can form from $A_k$ is $L_{k+1}-1$, which is less than $L_{k+2}$.~Therefore, $L_{k+2}$ cannot be formed as a non-consecutive sum of $A_{k+1}$, completing the inductive step.
\end{proof}

\begin{proof}[Proof of Lemma \ref{Lemma5}]
It suffices to show that every natural number of the form $L_{2m+1}+1$ is equal to only one non-consecutive sum of $A_{2m}$.~We proceed by strong induction. Note that $L_{3}+1$ is equal to only one non-consecutive sum of $A_2$, and $L_{5}+1$ is equal to only one non-consecutive sum of $A_4$.~This shows the base case.~Assume Lemma \ref{Lemma5} holds for all non-negative integers less than or equal to $m = k$.~Let $B$ be a non-consecutive subset of $A_{2k+2}$.~For sake of contradiction, suppose that the sum of the elements of $B$ is equal to $L_{2k+3}+1$ and that $B$ does not contain the term $L_{2k+2}$.~From Lemma \ref{Lemma3} the non-consecutive sums that we can form from $A_{2k+2}$ are the values from 0 to $L_{2k+3}+1$ inclusive, excluding $L_{2k+2}$.~This implies $B$ is a non-consecutive subset of $A_{2k+1}$.~From Lemma \ref{Lemma3} we have that the sum of the elements of $B$ must be less than or equal to $L_{2k+2}-1$.~Hence we have a contradiction, which implies $B$ contains the term $L_{2k+2}$.~From our inductive hypothesis, we know that $L_{2k+1}+1$ is equal to only one non-consecutive sum of $A_{2k}$.~Because $L_{2k+3}+1 = L_{2k+2} + \left(L_{2k+1} +1\right)$ and $B$ cannot contain both $L_{2k+2}$ and $L_{2k+1}$, this implies $L_{2k+3}+1$ is equal to only one non-consecutive sum of $A_{2k+2}$.~This completes the inductive step.
\end{proof}



\newcommand{\etalchar}[1]{$^{#1}$}

\ \\

\noindent MSC2010: 11B39

\end{document}